\documentclass{amsart}
\usepackage[english]{babel}
\usepackage{mathtools}
\usepackage{amssymb,amsfonts,thmtools}

\usepackage{hyperref}

\usepackage{float}

\renewcommand{\Re}{\operatorname{Re}}
\renewcommand{\Im}{\operatorname{Im}}

\newcommand{\z}[1]{z_{#1}}
\newcommand{\bz}[1]{\Bar{z}_{#1}}
\newcommand{\dz}[1]{\partial_{z_{#1}}}
\newcommand{\dw}{\partial_{w}}
\newcommand{\g}{\mathfrak{g}}
\newcommand{\gRe}{\mathfrak{g}_{0}^{\Re}}
\newcommand{\gIm}{\mathfrak{g}_{0}^{\Im}}
\newcommand{\gNil}{\mathfrak{g}_0^{\operatorname{Nil}}}

\newcommand{\gmu}[1]{\mathfrak{g}_{-\mu_{#1}}}

\newcommand{\bbZpos}{\mathbb{Z}^{>0}}
\newcommand{\bbZnon}{\mathbb{Z}^{\geq 0}}

\newcommand{\qaq}{\quad \text{and} \quad}

\DeclareMathOperator{\hol}{hol}

\newcommand{\bbZ}{\mathbb Z}

\newcommand{\bbR}{\mathbb R}
\newcommand{\bbC}{\mathbb C}

\numberwithin{equation}{section}

\newtheorem{theorem}{Theorem}[section]
\newtheorem{corollary}[theorem]{Corollary}
\newtheorem{lemma}[theorem]{Lemma}
\newtheorem{proposition}[theorem]{Proposition}
\newtheorem{remark}[theorem]{Remark}

\newtheorem{definition}[theorem]{Definition}
\newtheorem{example}[theorem]{Example}

\usepackage[para,online,flushleft]{threeparttable}

\title[Polynomial models in $\bbC^3$]{
Classification of polynomial models  without 2-jet~determination in $\bbC^3$}
\author[M. Kolář]{Martin Kol\'a\v r}
\author[P. Liczman]{Petr Liczman}
\author[F. Meylan]{Francine Meylan}

\address{M. Kol\'a\v r, P. Liczman: Department of Mathematics and Statistics, Masaryk University,
Kotlarska 2,  611 ~37 Brno, Czech Republic}
\email{mkolar@math.muni.cz}
\email{liczman@mail.muni.cz}

\address{ F. Meylan: Department of Mathematics,
University of Fribourg, CH 1700 Perolles, Fribourg}
\email{francine.meylan@unifr.ch}

\thanks{The first two authors were supported by the GAČR grant GA21-09220S}

\begin{document} 
\begin{abstract}
An intriguing phenomenon regarding
Levi-degenerate hypersurfaces is the existence of nontrivial infinitesimal symmetries with vanishing 2-jets at a~point. In this work we consider polynomial models of Levi-degenerate real hypersurfaces in $\bbC^3$
of finite Catlin multitype. Exploiting the structure of the corresponding Lie algebra, we characterize completely models without 2-jet determination, including an explicit description of their symmetry algebras.
\end{abstract}

\maketitle
\section{Introduction}
Our main motivation in this paper is the local equivalence problem for real hypersurfaces in complex space, as formulated by Poincar\'e in \cite{Po}. We are interested in the singular case, when the Levi form is degenerate at a point of the hypersurface.
This leads naturally to the study of holomorphically nondegenerate polynomial models, which generalize the model hyperquadrics from the classical Chern-Moser theory \cite{CM}.

The singularity of the Levi form prevents from using standard differential geometric techniques of Cartan, Chern, Tanaka \cite{C32,C2, T} based on vector bundles, since the rank of the form may change from point to point.
On the other hand, as it turns out, the analytic normal form approach to the equivalence problem, as developed by Moser and inspired by techniques coming from dynamics, is applicable to this setting \cite{KZ14, KZ15}.  In particular, Kol\'a\v r, Meylan and Zaitsev in \cite{KMZ14} studied a generalization of the Chern-Moser operator to polynomial models of finite Catlin multitype and described explicitly the structure of the Lie algebra of symmetries, as a graded Lie algebra.

An immediate consequence of the classical Chern-Moser theory is a sharp 2-jet determination result for local automorphisms of Levi nondegenerate hypersurfaces. While finite jet determination still holds on polynomial models in the setting of
 \cite{KMZ14}, the possible order of determinacy is not uniformly bounded, and there exist infinitesimal symmetries determined by jets of arbitrarily high order. In fact, this phenomenon presents a main obstruction to a generalization of the normal form constructions of \cite{CM}
and  \cite{K05, KZ14, KZ15}.
 
 In fact, existence or nonexistence of higher order symmetries on CR manifolds is a problem of its own great interest and has been studied intensively in various settings. There has been also intensive research on characterizing finite jet determination in general (\cite{LM, Z02} and references therein).

  Classifying all possible kernels of the Chern-Moser operator, and hence all Lie algebras of infinitesimal symmetries of such polynomial models, does not seem tractable in full generality. Note that $\mathbb C^2$ case is well understood and rather simple. It leads to three different types of models, two exceptional ones - circular and tubular, and the generic type. In contrast, already in $\mathbb C^3$ the problem becomes highly nontrivial.

In this paper we address the question of classifying the models which violate 2-jet determination,  that is, admitting   so called exotic symmetries. We also describe which other symmetries are compatible with an exotic symmetry. Observe that in this case $\dim \g \geq 4$. The results can be summarized by the given table \ref{tab:result_table}. (See Section \ref{sec:preliminaries} for notation.)

\begin{table}[H]
\scalebox{1}{
\begin{threeparttable}
\begin{tabular}{|r||r|r|r|r|c|c|}
\hline
$\g$ & $\g_t$& $\gRe$ & $\gIm$  &$\g_1$ &  Model & Reference\\
\hline
10&2&1&1&1& $\Im w = \Re\left(\bz1\z2^\alpha\right)$ &  \cite{KoMe11} \\
7&1&1&1&1& $\Im w = |\z1|^{2k}(\Re \z2)^m$ &  \autoref{gc+tub} \\
6&0&1&1&1& $\Im w = |\z1|^{2k}|\z2|^{2l}\left(\Re\z1^{\alpha}\z2^{\beta}\right)^m$ & \  \autoref{gc+re+im}\\
6&2&0&1&0& $\Im w = \Re \bz1\z2^{2\alpha-1} + |\z2|^{2\alpha}$&  \autoref{gc+tub} \\
5&2&0&0&0& $\Im w = \Re \bz1\z2^\alpha + Q(\z2, \bz2)$&  \autoref{gc+tub}\\
5&1&1&0&0& $\Im w = \Re\z2^\alpha\Re\bz1\z2^\alpha$ &  \autoref{gc+tub} \\
4&0&1&0&0& See \eqref{mod:gcre}         &  \autoref{gc+real} \\
4&0&0&1&0& See \eqref{mod:gcim}         & \autoref{gc+imag} \\
4&1&0&0&0& $\Im w = \Re\z2^\alpha\Re\bz1\z2^\alpha + Q(\z2, \bz2)$& \autoref{gc+tub} \\
\hline
\end{tabular}
\caption{}
\label{tab:result_table}
\end{threeparttable}
}
\end{table}
This table gives a classification  of the  models with respect to the dimension of the full symmetry algebra and, more precisely, with respect to   the dimensions of some of  its graded components that are compatible with the existence of an exotic symmetry.   The common $3$-dimensional part, $\g_c\oplus\langle W, E \rangle$ is omitted for all models in the table (here $W=\dw,$ where $w$ is the transversal coordinate, and $E$ is the Euler field). Also note that in $\bbC^3$ we always have $\dim\g_c\leq1$ \cite{KM16}.

 Note that the first model has an extra $2$-dimensional factor, $\g_n$, as recalled in~\eqref{models_with_gn}.
Further details on the models and their parameters are provided in the respective given  references.

The structure of the paper is as follows: In Section \ref{sec:preliminaries}, we recall the needed background, collect useful theorems and fix notation. 
 In Section \ref{sec:gc+rot}, we study the existence of an exotic symmetry together with rotations; in Section \ref{sec:gc+gt}, we explore exotic symmetries in combination with tubular symmetries. The general case, $\dim\g=3$, is treated in section \ref{sec:gc3dim}.

\section{Preliminaries}\label{sec:preliminaries}

Our goal is to characterize holomorphically nondenegerate polynomial models of finite Catlin multitype of Levi-degenerate hypersurfaces in $\bbC^3$ given by
\begin{equation}\label{model}
    M_P=\{ (z,w)\in \bbC^{2}\times\mathbb{C}\ |\ \Im w = P(z,\overline{z})\},
\end{equation}
where $P$ is a real valued weighted-homogeneous polynomial of weight one with respect to the Catlin multitype weights.

In this section, we fix notation and collect useful terminology and theorems from the literature. The statements are adapted for our needs in complex dimension 3.

\subsection{Catlin multitype}

We use nonnegative rational weights. We always have weight 1 for the complex transversal variable $w$. For tangential variables we need the following
\begin{definition}[Catlin, \cite{C}]
A weight is an $n$-tuple of rational numbers $\Lambda = (\lambda_1, \dots, \lambda_n)$ such that 
$$0\leq \lambda_k \leq \frac{1}{2} \qaq \lambda_k \geq \lambda_{k+1}$$

A weight $\Lambda$ is distinguished for $M$ if there exist local coordinates (at the origin) in which $M$ is given by
$$ \Im w = P(z,\Bar{z}) + o_\Lambda (1)$$
where $P$ is a polynomial of $\Lambda$-weighted degree 1 not containing any pluriharmonic terms. The $o_\Lambda (1)$ term vanishes up to the weighted order 1.

Denote $\Lambda_M=(\mu_1, \dots, \mu_n)$ the infimum of all distinguished weights in lexicographic order. The (Catlin) multitype is an $n$-tuple $(m_1, \dots, m_n)$ given by

$$
m_j=
\begin{cases}
\frac{1}{\mu_j}, \text{ if } \mu_j\neq 0,\\
\infty, \text{ if } \mu_j = 0.
\end{cases}
$$
\end{definition}

Multitype is clearly an invariant. The hypersurface $M$ has \emph{finite multitype} (or \emph{is of finite multitype}) if $m_n < \infty$, or equivalently, all weights are positive. 
For finite multitype hypersurfaces, the Lemma 3.1 from \cite{K10} implies that for each~$k$ there are nonnegative integers $a_1,\dots, a_k$ with $a_k>0$ such that
\begin{equation}\label{weights_monomial_condition}
    \sum_{j=1}^k a_j\lambda_j = 1.
\end{equation}
Thanks to this condition, there are only finitely many weights with $\lambda_n>\epsilon>0$. The infimum is therefore attained as a distinguished weight, which, in turn, provides us with \emph{multitype coordinates} and a model for $M$ given by
$$M_P = \{ \Im w = P(z,\Bar{z}) \}.$$
This model is not unique (as multitype coordinates are not unique), but it is shown in \cite{K10} that all such models are biholomorphically equivalent; hence, the model is an invariant.

\subsection{The symmetry algebra and its graded parts}

As an application of the generalized Chern-Moser theory, it is proved in \cite{KMZ14} that the Lie algebra of infinitesimal symmetries of a finite multitype holomorphically nonnedegerate model hypersurface $M_P$, denoted by
$\g=\hol(M_{P},0)$, admits a weighted grading given by
\begin{equation}
\g = \g_{-1} 
\oplus \g_{- \mu_1}\oplus \g_{-\mu_2} 
\oplus \g_{0}
\oplus \g_c 
\oplus \g_n
\oplus \g_{1},
\label{algebra}
\end{equation}
where the weights are given by the multitype. We briefly recall description of each factor: $\g_{-1}=\langle W \rangle $, where $W=\dw$. Symmetries of weights in the interval $(-1,0)$ are called \emph{tubular}. The factor $\g_c\oplus\g_n$ consists of symmetries of weights strictly between $0$ and $1$. Vector fields in $\g_c$ commute with $W$, while fields in $\g_n$ do not. Nontrivial elements of $\g_c$ are called \emph{exotic symmetries} (or \emph{generalized rotations}). In $\bbC^3$ a nontrivial $\g_n$ is rare, indeed in \cite{KM19} it has been shown that occurs only for models equivalent to 
\begin{align}\label{models_with_gn}
    \Im w &= \Re(\z1\bz2^l),\text{\, where }l>1, & \dim \g &= 10, \\
    \intertext{or to }
    \Im w &= |\z1|^2\pm |\z2|^{2l},\text{\, where }l>1, & \dim \g &= 9. \label{mod:9dim}
\end{align}
Vector fields that commute with $W$ are called \emph{rigid}; in other words, this means their coefficient functions are independent of the variable $w$.
The factor $\g_0$ consists of vector fields of weight zero. There is always present a unique (up to a real multiple) non-rigid field called the (multitype) Euler field $E$, which takes the following form in every multitype coordinates
\begin{equation}
E = w \dw + \mu_1 \z1\dz1+\mu_2\z2\dz2.
\end{equation}
The rigid elements of $\g_0$ are called \emph{rotations}. There exist multitype coordinates such that each rotation is linear and every tubular symmetry regular, i.e., nonvanishing at zero \cite{KMZ14}. A~rotation $Y$ can be decomposed further as $Y = Y^{\Re} + Y^{\Im} + Y^{\operatorname{Nil}}$, where each term is also a symmetry (\cite{KMZ14}, \cite{KM21}). This is obtained by putting $Y$ into its Jordan normal form, $Y^{\Re}$ is then the real diagonal part, similarly with the other two terms. If $Y=Y^{\Re}$ we call it a \emph{real rotation}. Analogously we have \emph{imaginary} and \emph{nilpotent} rotations. As a shorthand, we denote sets of these rotations by $\gRe, \gIm$ and $\gNil$. Notice these sets are not even subspaces, but thanks to a standard fact from linear algebra, we have the following lemma
\begin{lemma}
    If all rotations are diagonalizable and commute, then they are diagonalizable simultaneously. Consequently,
    these sets are subalgebras and we can write $$\g_0 = \gRe \oplus \gIm \oplus \langle E \rangle.$$
\end{lemma}
 We will see in \autoref{nonil} that $\g_c$ is not compatible with a nilpotent rotation. Thanks to the following proposition, due to the first author (Lemma 4.2, \cite{K21}), we can assume throughout the entire paper that $\g_0$ can be diagonalized. 
\begin{lemma} \label{lem:4.2}
    Let $\g_0$ be the graded component of weight $0$ of the Lie algebra $\g$ of a polynomial model $M_P.$ Suppose that $\g_0$ is nonabelian (with respect to the commutator of vector fields). Then either $M_P$ admits a nilpotent rotation, or $M_P$ is equivalent to
    \begin{equation}\label{mod:quadric^m}
    \Im w = \left( |\z1|^2 + |\z2|^2 \right)^m
    \end{equation}
    for some integer $m>1$. In the latter case $\dim \g = 7$ and $\dim \g_0=5$.
\end{lemma}

\begin{definition}
    A vector field $Y$ satisfying 
        $\Re(YP) = P$
    is called \emph{real reproducing field}. If the vector field satisfies $YP=P$ we call it \emph{complex reproducing field}.
\end{definition}
The Euler field is clearly a real reproducing field, complex reproducing fields are relevant for the $\g_1$ factor by \autoref{g1}.
The following lemmas are straightforward and can be proved by an useful, albeit elementary, observation. Action of a diagonal linear field reproduces monomials in the following sense: $Y(z^a\bz{}^b) = cz^a\bz{}^b$ for some number $c$ that can be thought of as a weight of the monomial.
\begin{lemma}\label{real_rot}
    Let $Y$ be a diagonal linear vector field $Y=\lambda_1\z1\dz1+\lambda_2\z2\dz2$, where $\lambda_1,\lambda_2\in\bbR$. $Y$ is a real rotation of the model $M_P$, given by \eqref{model} if and only if each monomial in $P$ is of  weight $0$ with respect to the weights $(\lambda_1, \lambda_2)$, i.e. for each monomial 
    $\z1^{a_1}\z2^{a_2}\bz1^{b_1}\bz2^{b_2}$
    in $P$ 
    we have
    \begin{equation*}
        \lambda_1(a_1+b_1) + \lambda_2(a_2+b_2) = 0.
    \end{equation*}
\end{lemma}
There is a similar condition 
for the existence of imaginary rotations:
\begin{lemma}\label{imag_rot}   
    Let $Y$ be a diagonal linear vector field $Y=i(\lambda_1\z1\dz1+\lambda_2\z2\dz2)$, where $\lambda_1,\lambda_2\in\bbR$. $Y$ is an imaginary rotation of model $M_P$ \eqref{model} if and only if for each monomial $\z1^{a_1}\z2^{a_1}\bz1^{b_1}\bz2^{b_2}$ in $P$ we have
    \begin{equation*}
        \lambda_1(a_1-b_1) + \lambda_2(a_2-b_2) = 0.
    \end{equation*}
\end{lemma}

The following characterization follows from \cite{KMZ14} and \cite{JM}.
\begin{lemma}\label{g1}
    For a model \eqref{model} the following are equivalent:
    \begin{enumerate}
    \item The model admits a nontrivial $\g_1$.

    \item There exists a complex reproducing vector field for $P$. 
   
    \item There exist multitype coordinates and real numbers $\lambda_1, \lambda_2$ such that for each monomial term $\z1^{a_1}\z2^{a_2}\bz1^{b_1}\bz2^{b_2}$ in $P$ we have
    \begin{equation*}
        \lambda_1 a_1 + \lambda_2 a_2 = \lambda_1 b_1 + \lambda_2 b_2 = 1.
    \end{equation*}
    \end{enumerate}
\end{lemma}

\begin{remark}
    In coordinates given by (3) the complex reproducing field is expressed as
        $Y=\lambda_1\z1\dz1+\lambda_2\z2\dz2.$
    Note that $iY$ is then an imaginary rotation. Recall that the polynomial $P$ and the model satisfying these conditions are called \emph{balanced}.
\end{remark}

\subsection{Exotic symmetries}\label{sec:exotics}
The following is a characterization of exotic symmetries from \cite{KM16}.
\begin{definition}\label{xpair}
Let $X$ be a weighted homogeneous vector field.
A pair of finite sequences of holomorphic  weighted homogeneous polynomials $\{U_0, \dots, U_m\}$ and  $\{V_0, \dots, V_m\}$ is called an $X$-pair of  chains (of length $m+1$), if
\begin{equation}
X(U_m)=0, \ \ X(U_j) =A_j U_{j+1}, \ \ j=0, \dots, m-1,
\end{equation}
\begin{equation}
X(V_m)=0, \ \ X(V_j) =B_j V_{j+1}, \ \ j=0, \dots, m-1,
\end{equation}
where $A_j$, $B_j$ are nonzero complex numbers that satisfy
\begin{equation} A_j= - \overline B_{m-j-1}, \ \ j=0, \dots, m-1.
\end{equation}  
\end{definition}

\begin{theorem} \label{chainsum}
    Let $M_P$  be a holomorphically nondegenerate hypersurface  given by  \eqref{model},  which admits a nontrivial $X\in \g_c$. Then $P$ can be decomposed in the following way

\begin{equation}P_C=\sum_{j=1}^M T_j,\label{}
\end{equation} where each $T_j$ is given by
\begin{equation}\label{chainsumeq}
 T_j =\Re (\sum_{k=0}^{m_j}
  {U^j_{k}}
  {\overline {{V^j_{m_j -k}}}}),
  \end{equation}
where $\{{ {{U^j_{0}}, \dots, {U^j_{m_j}} }}\}$  and  $\{{ {{V^j_{0}}, \dots, {V^j_{m_j}} }}\}$ are  $X$-pairs of chains.
\end{theorem}

\section{Exotic symmetries and rotations}\label{sec:gc+rot}
In this section, 
we first prove that $\g_c$ is incompatible with a nilpotent symmetry.
Then a careful analysis of monomial $X$-pairs yields to Theorems \ref{gc+real} and \ref{gc+imag} that characterize existence of rotations.

\begin{theorem}\label{nonil}
    If a holomorphically nondegenerate model \eqref{model} possesses a nilpotent rotation, then there are no exotic symmetries, i.e. $\g_c=0$.
\end{theorem}
\begin{proof}
    Let $Y$ be a nilpotent rotation. We can choose multitype coordinates such that $Y=\z1\dz2$. It follows that $\mu_1=\mu_2$. In these coordinates the exotic symmetry $X\in\g_c$ of weight $\nu>0$ is of the form 
    \begin{equation*}
        X = f\dz1+g\dz2, 
    \end{equation*}
    where $f,g$ are holomorphic weighted homogeneous polynomials of degree $\nu+\mu_1$. Since in $\bbC^3$ $\dim \g_c \leq 1$ it follows, that $[X,Y]=kX$ for some $k\in\bbR$. Hence 
    \begin{equation*}
    [X,Y] = -\z1f_{\z2}\dz1+(f-\z1g_{\z2})\dz2=kX.
    \end{equation*}
    By collecting terms we obtain a system
    \begin{align}
        -\z1f_{\z2} &= kf, \label{fstnil}\\
        f-\z1g_{\z2} &= kg. \label{sndnil}
    \end{align}
    
    We consider two cases.
    
    \textbf{Case I}: $k=0$. It follows $f=a\z1^m$ and $g=a\z1^{m-1}\z2 + b\z1^m$, where $m\geq2$, $a,b\in\bbC$. If $a=0$, we get $X=b\z1^m\dz2$, but that would mean $X=b\z1^{m-1}Y$ which cannot be a symmetry. Indeed, if it were, then
    \begin{align*}
    0=\Re(XP)=\Re(b\z1^{m-1}YP) &= \Re(b\z1^{m-1})\Re(YP) - \Im(b\z1^{m-1})\Im(YP) \\
                                &= -\Im(b\z1^{m-1})\Im(YP),
    \end{align*}
    where $\Re(YP)=0$ since $Y$ is a symmetry and $\Im(YP)\neq 0$ from holomorphic nondegeneracy. Hence $b=0$ and $X=0$. If $a\neq 0$, we can change coordinates by re-scaling in $\z1$ in order to remove $a$ and obtain
    $$X = \z1^{m-1}(\z1\dz1+(\z2+c\z1)\dz2),$$
    for some $c\in\bbC$.
    This $X$ cannot be an exotic symmetry, since it does not annihilate any nonzero polynomial, i.e. it does not admit any $X$-pair. To see this, denote $L$ the linear part of $X$ in parentheses and $Q$ a polynomial in $\z1,\z2$ such that $L(Q)=0$. Focus on the term of the highest degree $\beta$ in $\z2$ and compute
    \begin{equation*}
    L(a_{\alpha\beta}\z1^\alpha\z2^\beta) = a_{\alpha\beta}(n\z1^\alpha\z2^\beta+c\beta\z1^{\alpha+1}\z2^{\beta-1}).
    \end{equation*}
    We notice the first term is of the same degree and cannot be eliminated by other terms of lower degree.
    
    \textbf{Case II}: $k\neq0$. We show the system does not have a nonzero solution. The only solution to \eqref{fstnil} is $f=0$, which can be again seen by focusing on the term of $f$ of the highest degree in $\z2$. This term appears on the right-hand side (because $k\neq 0$) but not on the left-hand side. Substituting $f=0$ into \eqref{sndnil} gives equation of the same form, hence $f=g=0$ is the only solution.
\end{proof}

The existence of rotation has a strong consequences for the form of the exotic symmetry.
\begin{lemma}\label{diagrots}
Suppose that the holomorphically nondegenerate model \eqref{model} admits 
a rotation 
$Y\in\g_0$  together with  an exotic symmetry $X\in\g_c.$
 Then there are coordinates in which $Y$ is diagonal and $X$ monomial. More precisely, there are coordinates such that
\begin{gather}\label{diagrots_fields}
    Y = \lambda_1\z1\dz1+\lambda_2\z2\dz2
    \\ \label{field:monomial_exotic}
    X = i\z1^{\alpha}\z2^{\beta}(q\z1\dz1-p\z2\dz2),\ \text{or}\  X = i\z1^{\beta}\dz2,\ \text{or}\  X = i\z2^{\alpha}\dz1,
\end{gather}
where  $\lambda_1,\lambda_2\in\bbC$
and $p,q,\alpha,\beta$ are nonnegative integers.
\end{lemma}
\begin{remark}
    The case when $X=i\z1^{\beta}\dz2$ can be understood as the two term form of $X$ with $\alpha=-1$, holomorphicity of $X$ then implies $p=0$. Similarly the remaining field corresponds to $\beta=-1$.
\end{remark}
\begin{proof}
By \autoref{nonil} $Y$ is diagonalizable, hence we can choose some multitype coordinates such that $Y=\lambda_1 z_1\dz1+\lambda_2 z_2\dz2$, where  $\lambda_1, \lambda_2 \in \bbC$. Denote $0<\nu<1$ the weight of $X$.
In these coordinates $X=f_1(z) \dz1 + f_2(z) \dz2$, where $f_j$'s are weighted homogeneous holomorphic polynomials of weighted degree $\nu+\mu_j$. The Lie bracket $[Y,X]$ is a symmetry of weight $\nu>0$. 
Since $\dim\g_c = 1$, the only possibility is 
\begin{equation}\label{lie_condition}
    [Y,X]=kX
\end{equation}
for some $k\in \bbR$. We will prove that $f_j$'s are monomials. Let us investigate a general monomial term $z_1^\gamma z_2^\beta$ in $f_1$ and find constraints on $\gamma,\beta$. Computing the Lie bracket gives us
\begin{equation*}
[Y,z_1^\gamma z_2^\beta\dz1] = \left(\lambda_1 \gamma z_1^\gamma z_2^\beta + \lambda_2 \beta z_1^\gamma z_2^\beta - \lambda_1 \gamma z_1^\gamma z_2^\beta\right)\dz1 = 
\left(\lambda_1(\gamma-1)+\lambda_2\beta\right)\z1^\gamma\z2^\beta\dz1.
\end{equation*}
This means that possible terms in $X$ do not interfere with each other when taking the bracket with $Y$. The Lie bracket condition \eqref{lie_condition} leads to a linear equation
\begin{align}
    \lambda_1 (\gamma-1) +\lambda_2 \beta &= k,\label{fst}\\
    \mu_1(\gamma-1) + \mu_2\beta &= \nu,
\end{align}
where the second relationship comes from the homogeneity of $X$.

The system is regular. Indeed, write $Y=Y^{\Re}+Y^{\Im}$. If $Y^{\Re}\neq0$, its coefficient vector $(\Re{\lambda_1}, \Re{\lambda_2})$ is independent of $(\mu_1, \mu_2)$, since the Euler field reproduces polynomials. If $Y=Y^{\Im}$ is an imaginary rotation, i.e. $\lambda_1,\lambda_2\in i\bbR$, then $k=0$, because $k$ is real. Since $\nu\neq 0$, it follows that, again, the rows are independent. In both cases there is a unique solution $\gamma, \beta$ and therefore $f_1 = \sigma_1 z_1^\gamma z_2^\beta$ for some~$\sigma_1\in\bbC$.

Similarly one obtains $f_2(\z1,\z2) = \sigma_2 z_1^\alpha z_2^\delta$ satisfying 
\begin{align}
    \lambda_1 \alpha +\lambda_2 (\delta-1) &= k \label{snd},\\
    \mu_1\alpha + \mu_2(\delta - 1) &= \nu.
\end{align}
Combining these two systems we get
\begin{align}
    \lambda_1 (\gamma - 1 - \alpha) +\lambda_2 (\beta - \delta + 1) &= 0,\\
    \mu_1 (\gamma - 1 - \alpha) +\mu_2 (\beta - \delta + 1) &= 0.
\end{align}
Again by uniqueness $\gamma = \alpha + 1$, $\beta+1=\delta$. Thus the symmetry $X$ is of the form
\begin{equation}
    X=f_1 \dz1 + f_2 \dz2 = z_1^\alpha z_2^\beta(\sigma_1 z_1 \dz1 + \sigma_2 z_2 \dz2).
\end{equation}

It is worth noting that the case $\gamma = 0$ (or $\delta=0$) forces $\sigma_2=0$ (or $\sigma_1=0$), yielding a vector field $X$ of the form
\begin{align*}
    X=\sigma\z2^\beta\dz1 \quad \text{or}\quad X=\sigma\z1^\alpha\dz2.
\end{align*}
Now we determine $\sigma_1,\sigma_2$. Since $X$ is an exotic symmetry, it annihilates the last polynomial of any $X$-pair of chains. Write $L=\sigma_1 z_1 \dz1 + \sigma_2 z_2 \dz2$ and $L(U)=0$ for some weighted homogeneous polynomial $U$. The vector field $L$ reproduces monomials in the following sense 
$L(\z1^p\z2^q)=(\sigma_1p+\sigma_2q)\z1^p\z2^q$, hence $\sigma_1p+\sigma_2q=0$ for each monomial in $U$, since they cannot cancel each other in $X(U)$. Also the space of weighted homogeneous solutions of $X(U)=0$ of given weight is at most $1$-dimensional, implying $U$ is a monomial. Additionally $(\sigma_1, \sigma_2)=\theta(q,-p)$ for some $\theta \in \bbC$. It follows $X=\theta z_1^\alpha z_2^\beta(q z_1 \dz1 - p z_2 \dz2)$. By re-scaling in one of the variables, we arrive at the convenient form
\begin{equation*}\label{monomialfield}
    X=iz_1^\alpha z_2^\beta(q z_1 \dz1 - p z_2 \dz2).\qedhere
\end{equation*}
\end{proof}

From now on, we  will work in the coordinates given by Lemma \eqref {diagrots}. 
The proof of the Lemma offers useful and crucial parametrization collected in  the following 
\begin{corollary}
For a monomial vector field \eqref{field:monomial_exotic} there exists a monomial solution~$Q$ to $X(Q)=0$ of minimal (weighted) degree. Any other solution to $X(U)=0$ is of the form $U=cQ^K$, where $K$ is nonnegative integer and $c\in\bbC$. In the following, whenever we reference the vector field \eqref{field:monomial_exotic}, we denote $Q$ this minimal solution $Q=\z1^p\z2^q$. Note that $p,q$ are coprime.

\end{corollary}
We have to analyze every possible hypersurface that arises from \autoref{chainsum} applied to the exotic symmetry $X$ of the form \eqref{field:monomial_exotic}. The first step is to describe all possible $X$-pairs of chains.

\begin{definition}
    Let $X$ be a monomial vector field of the form \eqref{field:monomial_exotic}, $Q$ its minimal solution and denote $T=\z1^\alpha\z2^\beta$. Additionally, let $K,N,m$ be nonnegative integers, $K$ positive, such that $Kp>m\alpha$ and $Kq>m\beta$. We define a \emph{pure} $X$-pair to be a monomial $X$-pair given by
    \begin{equation*}
    U_{m-j} = \frac{1}{j!}\frac{Q^K}{T^j} = \frac{1}{j!}\z1^{Kp-j\alpha}\z2^{Kq-j\beta}
    \end{equation*}
    and
    \begin{equation*}
    V_{m-j} = U_{m-j}Q^N = \frac{1}{j!}\frac{Q^K Q^N}{T^j} = \frac{1}{j!}\z1^{Kp-j\alpha}\z2^{Kq-j\beta}\z1^{Np}\z2^{Nq}
    \end{equation*}
    for $j=0,1,\cdots,m$.
\end{definition}
\noindent
Note that the pair just defined in indeed an $X$-pair. It satisfies
    \begin{align*}
            &X(U_m) = 0,                       & &X\left(V_m\right) = 0, \\
    &X(U_{m-j}) = i\left(p\beta-q\alpha\right)U_{m-j+1}, & &X(V_{m-j}) = i\left(p\beta-q\alpha\right)V_{m-j+1}.
    \end{align*}
    Here the number $i\left(p\beta-q\alpha\right)$ acts as all the constants $A_j,B_j$. The utility of pure $X$-pairs is a normalization; each monomial $X$-pair differs from a pure $X$-pair only by coefficients.
It is useful to explicitly compute the \emph{chain sum} as in \eqref{chainsumeq} for any pure $X$-pair. The corresponding model is
\begin{multline}\label{monomialsum}
    \Im w  = \Re\left(\sum_{j=0}^{m} U_j\overline{V}_{m-j} \right)= \\
    =\Re\left(\sum_{j=0}^{m} \frac{1}{j!(m-j)!}\z1^{Kp-j\alpha}\z2^{Kq-j\beta}\bz1^{Kp-(m-j)\alpha}\bz2^{Kq-(m-j)\beta}\bz1^{Np}\bz2^{Nq}\right)  =\\
    =\frac{2^{m}}{m!}\Re\left(\left(\z1^p\z2^q\right)^N\right)|\z1|^{Kp-m\alpha}|\z2|^{Kq-m\beta}\left(\Re\left(\z1^\alpha\z2^\beta\right)\right)^m. 
\end{multline}

Requiring this model to be holomorphically nondegerate forces $(p\beta-q\alpha)\neq 0$. This is an important observation. In other words $(\alpha, \beta)$ is independent of $(p,q)$. 
Indeed, if they were dependent, the right-hand side would be a function of $\z1^p\z2^q$ and $Z=q\z1\dz1-p\z2\dz2$ would be a complex tangent field.

The following lemma is also useful for analyzing all the possible $X$-pairs of chains, where $X$ is not necessarily monomial. It computes how the corresponding chain sum of an $X$-pair changes if the pair is multiplied by constants.
\begin{lemma}\label{constantmultiple}
    Let $\{U_0,\dots,U_m\}$ and $\{V_0, \dots,V_m\}$ be an $X$-pair of chains. Let $c_j, d_j$ be nonzero complex numbers such that sequences $U'_j := c_j U_j$ and $V'_j := d_j V_j$ also form an $X$-pair of chains. Then there exists a nonzero complex number $\tau$ such that the corresponding chain sum of $\{U'_0,\dots,U'_m\}$ and $\{V'_0,\dots,V'_m\}$ amounts to
    \begin{equation}
        \Re\left(\sum_{j=0}^m U_j'\overline{V'}_{m-j}\right)=
        \Re\left(\tau\sum_{j=0}^m U_j\overline{V}_{m-j}\right).
    \end{equation}
\end{lemma}
\begin{proof}
    Denote the constants of both $X$-pairs by $A_j, B_j$ and $A'_j, B'_j$ as in the Definition \ref{xpair}. Applying the field gives
    \begin{displaymath}
        X(U'_j)=X(c_jU_j)=c_jA_jU_{j+1} = \frac{c_jA_j}{c_{j+1}}U'_{j+1}.
    \end{displaymath}
    An analogous computation for $V'_j$ gives us
    \begin{equation*}
        A'_j = \frac{c_j}{c_{j+1}}A_j, \ \ B'_j = \frac{d_j}{d_{j+1}}B_j,
    \end{equation*}
    for $j=0,\dots,m-1$.
    The antihermitian condition $A'_j=-\overline{B}'_{m-j-1}$ translates to
     $c_j\overline{d}_{m-j}=c_{j+1}\overline{d}_{m-j-1}$. By induction, this expression does not depend on the index $j$ and is therefore a constant. This is the number $\tau$, since it is precisely the coefficient appearing in the corresponding chain sum. Indeed
     \begin{displaymath}
         U_j'\overline{V'}_{m-j} = c_j\overline{d}_{m-j}U_j\overline{V}_{m-j} = \tau U_j\overline{V}_{m-j}.\qedhere
     \end{displaymath}
\end{proof}
\begin{remark}\label{normalizationofchains}
    The previous lemma offers an opportunity to simplify \autoref{xpair} of an $X$-pair of chains. Indeed, we can always find appropriate $c_j$ and $d_j$ such that the constants for the new $X$-pair are all equal to any nonzero imaginary number $A'_j=B'_j=ic\in i\bbR^*$. However, the corresponding chain sum can change. But that is easy to fix -  divide one of the chains by $\tau$.
\end{remark}

The following lemma deals with a general (nonmonomial) $X$-pair for our monomial vector field of the form \eqref{field:monomial_exotic}. It turns out they can be thought of as a combination of monomial pairs.
\begin{lemma}\label{decomposition}
    Let $\{U_0,\dots,U_m\}$ and $\{V_0, \dots,V_m\}$ be a nonmonomial $X$-pair of chains, where $X$ is a monomial exotic symmetry given by \eqref{field:monomial_exotic}. Then there exist several monomial $X$-pairs of chains $\{U_0^t,\dots,U_{m_t}^t\}$ and $\{V_0^t,\dots,V_{m_t}^t\}$ such that
    \begin{equation}\label{sumdecomposition}
        \Re\left(\sum_{j=0}^m U_j\overline{V_{m-j}}\right) = 
        \sum_t \Re\left(\sum_{j=0}^{m_t} U_j^t\overline{V_{m_t-j}^t}\right).
    \end{equation}
	Moreover, we get an explicit description of each monomial \ref{coolA}.
\end{lemma}
\begin{proof}
    By the previous Remark, we can assume 
    \begin{equation}\label{reccurence}
        X(U_{j})=i(p\beta-q\alpha)U_{j+1}\quad \text{for}\quad j=0,\dots,m-1.
    \end{equation}
    Denote $U_m = a_mQ^{S_m}$ for some positive integer $S_m$ and $a_m\in\bbC^*$. Denote again $T=\z1^\alpha\z2^\beta$.
    Using the fact that $X$ is monomial, we solve this recurrence relation as
    \begin{equation}\label{coolA}
        U_j =  \sum_{k=j}^m A_j^k,\quad \text{ where }\quad 
        A_j^k=\frac{a_k}{(k-j)!}\frac{Q^{S_k}}{T^{k-j}}.
    \end{equation}
    Here $S_j\in\bbZ^+$ and $a_j\in\bbC$. The polynomial $U_j$ has to be weighted-homogeneous, hence there are conditions fixing $S_j$, but we do not need them explicitly. In a case where $S_j$ is not a positive integer, we do not obtain a polynomial. This situation is conveniently included in $a_j=0$. 
    
     The sequences of monomials $\{A_0^k,\dots,A_k^k\}$ satisfy
     \begin{equation*}
         X(A_k^k) = 0 \quad \text{and}\quad 
         X(A_j^k)=i(p\beta-q\alpha)A_{j+1}^k
         \ \text{for}\ j=0,\dots,k-1.
     \end{equation*}
     These are almost chains, but the polynomials are all zero if $a_k=0$. 
     Analogously we identify monomials $B_j^l$ such that
     \begin{equation}\label{coolB}
        V_{m-j} =  \sum_{l=m-j}^m B_{m-j}^l,\quad \text{ where }\quad 
        B_{m-j}^l=\frac{b_l}{(l+j-m)!}\frac{Q^{R_l}}{T^{l+j-m}}.
    \end{equation}
    After introducing this notation we can substitute sums \eqref{coolA} and \eqref{coolB} into \eqref{sumdecomposition} and easily distribute.
    \begin{align*}
        \sum_{j=0}^m U_j\overline{V_{m-j}} &= 
        \sum_{j=0}^m \left(\sum_{k=j}^m A_j^k\right)\left( \sum_{l=m-j}^m \overline{B_{m-j}^l}\right) = \\
        &= \sum_{k,l} \left( \sum_{j=m-l}^k A_j^k \overline{B_{m-j}^l}\right).
    \end{align*}
    Naturally, the inner sums could be zero. Either the sequences are zero themselves or $k+l<m$. Terms in nonzero sums are precisely the $X$-pairs we are after. They are of length $k+l+1-m$ and there are less than $(m+1)^2$ of them.
\end{proof}

Finally, using \eqref{monomialsum} and Lemmas \ref{constantmultiple} and \ref{decomposition} we obtain the following Theorem, which gives  a description of a  model  admitting a  monomial exotic symmetry.

\begin{theorem}\label{tutu}
 Assume that the holomorphically nondegenerate model $\eqref{model}$ possesses a monomial  exotic symmetry $X$ given by \eqref{field:monomial_exotic}. Then it  can be expressed as
\begin{equation}\label{mod:monomial_gc}
    \Im w = \sum_{N,K,m} \Re\left(\tau_{N,K,m}\left(\z1^p\z2^q\right)^N\right)|\z1|^{2k}|\z2|^{2l}\left(\Re\z1^\alpha\z2^\beta\right)^m,
\end{equation}
where $N,K,m, k,l$ are nonnegative integers and $k=Kp-m\alpha$, $l=Kq-m\beta$. The coefficients $\tau_{N,K,m}$ are some nonzero complex numbers. The polynomial has a weighted degree $1$. The special case of $\alpha = -1$ is then of the form
\begin{equation}\label{mod:monomial_gc2}
    \Im w = \sum_{N,K,m}
    \Re\left(\tau_{N,K,m}\z2^N\right)|\z2|^{2(K-m\beta)}\left(\Re\bz1\z2^\beta\right)^m.
\end{equation}
For $\beta=-1,$ the variables are interchanged. 
\end{theorem}

Theorem \ref{tutu} will be a key point when the model admits a (diagonal) rotation, the reason being that under this additional assumption, there exist multitype coordinates for which the exotic symmetry is, in fact, monomial.
\subsection{ Exotic symmetry together with real rotations}
In this subsection we analyze the situation where our model \eqref{model} admits both $\g_c$ and $\gRe$. 
After a possible change of multitype coordinates we can assume the real rotation to be diagonal; $Y=\lambda_1\z1\dz1+\lambda_2\z2\dz2$, where $\lambda_1,\lambda_2\in\bbR$. By the Theorem $\ref{diagrots}$ the exotic symmetry $X$ is monomial in these coordinates and given by \eqref{field:monomial_exotic}. The results of previous section apply and the model is again equal to \eqref{mod:monomial_gc}.

The additional existence of the rotation $Y\in \gRe$ imposes some new constraints on the parameters $N,k,l,m$ in the sum \eqref{mod:monomial_gc}. Let us explore them.

Using \autoref{real_rot}, $Y$ is a symmetry of \eqref{model} if and only if for each monomial $\z1^{a_1}\z2^{b_1}\bz1^{a_2}\bz2^{b_2}$ in $P$ we have 
\begin{equation*}
\lambda_1(a_1+a_2)+\lambda_2(b_1+b_2)=0.
\end{equation*}
Additionally, the condition of weighted-homogeneity is
\begin{equation*}
\mu_1(a_1+a_2)+\mu_2(b_1+b_2)=1.
\end{equation*}

We have already seen this is a regular system in the proof of Lemma \ref{diagrots}, hence it has a unique solution. In other words, the total degrees in variables $\z1,\bz1 $ and $\z2, \bz2$ are constant across all monomials. Denote these degrees $k_1$ and $k_2$, they satisfy the following system in $K,N$ and $m$.
\begin{align*}
    2pK + pN - \alpha m = k_1, \\
    2qK + qN - \beta m = k_2.
\end{align*}
From the kernel of this system we observe that $m$ and $2K+N$ are actually fixed. Indeed
$$
\left(\begin{array}{ccc}
   2p  & p & -\alpha \\
   2q  & q & -\beta 
\end{array}\right) \sim
\left(\begin{array}{ccc}
   2  & 1 & 0 \\
   0 & 0 & 1 
\end{array}\right)
$$
where we used only $q\alpha-p\beta\neq 0$.

This discussion proves most of the assertions in the following characterization.
\begin{theorem}\label{gc+real}
Let $M_P$ be a holomorphically nondegenerate  model  given by \eqref{model} with $\g_c\neq 0$ and $\gRe\neq 0 $. Then there exist positive integers $C$ and $m$ such that the model is locally equivalent to
\begin{equation}
    \Im w = \sum_{2K+N = C} \Re\left(\tau_{K,N}\left(\z1^p\z2^q\right)^N\right)|\z1|^{2(Kp-m\alpha)}|\z2|^{2(Kq-m\beta)}\left(\Re\z1^\alpha\z2^\beta\right)^m,
\label{mod:gcre}
\end{equation}
where $\tau_{K,N}$ are nonzero complex numbers; the sum is taken over some nonnegative integer solutions to $2K+N = C$. Moreover, if $\g_1 = 0$, then at least one $N$ is nonzero. In these coordinates the real rotation is given by 
\begin{equation} \label{real_rotation}
    Y = (Cq-m\beta)\z1\dz1-(Cp-m\alpha)\z2\dz2.
\end{equation}
\end{theorem}
\begin{proof}
The equation \eqref{mod:gcre} is precisely \eqref{mod:monomial_gc} with fixed $m$. The only part that remains to be argued is the condition $\g_1=0$. If the first factor $\Re\tau Q^N$ would not appear, then the model would be \ref{mod:gc6dim}. But such a  model admits a nontrivial $\g_1$, since it is balanced.

The formula for the rotation $Y$ is an application of the Lemma \ref{real_rot}. Indeed, the total degree in $\z1,\bz1$ is
\begin{equation*}
Np+2(Kp-m\alpha)+m\alpha = Cp - m\alpha,
\end{equation*}
and similarly for degree in $\z2,\bz2$.
\end{proof}

The following example confirms that the case described in the previous theorem does indeed happen.
\begin{example}
A hypersurface in $\bbC^3$ given by
        $$\Im w = \Re(\z1^3\z2)|\z1|^2\Re(\z1^2\z2)$$
    is holomorphically nondegenerate, admits $\g_c$ and $\gRe$, but $\g_1=0$. This example corresponds to parameters $p=3$, $q=1$, $\alpha=2$, $\beta=1$, $m=1$, $K=1$ and $N=1$
\end{example}

\subsection{Exotic symmetry together with imaginary rotations} 
Now we need to check what is forced by the existence of an imaginary rotation.
\begin{theorem}\label{gc+imag}
  Let $M_P$ be a holomorphically nondegenerate  model  given by \eqref{model} with $\g_c\neq 0$ and $\gIm\neq 0$. Then the model is locally equivalent to
    \begin{equation}
    \Im w = \sum_{(K,m)\in R} r_{K,m}|\z1|^{2(Kp-m\alpha)}|\z2|^{2(Kq-m\beta)}\left(\Re\z1^\alpha\z2^\beta\right)^m,
\label{mod:gcim}
\end{equation}
where $r_{K,M}$ are real numbers and 
\begin{equation*}
R = \left\{(K,m)\,|\,\mu_1(2Kp-m\alpha)+\mu_2(2Kq-m\beta)=1,\,K\in\mathbb{Z}^{>0},m\in\mathbb{Z}^{\geq0}\right\}.
\end{equation*}
Moreover, if $\g_1 = 0$, then the sum has at least two terms.
In these coordinates the imaginary rotation is given by 
\begin{equation}\label{imaginary_rotation}
    Y = i(\beta\z1\dz1-\alpha\z2\dz2).
\end{equation}
\end{theorem}
\begin{proof}
    The index set $R$ is nothing but a description of the weighted-homogeneity of the model. 
     Denote the imaginary rotation in its diagonal coordinates $Y=i(\lambda_1\z1\dz1+\lambda_2\z2\dz2)$. The exotic symmetry $X$ is again monomial given by \eqref{field:monomial_exotic}. We will use the characterization of imaginary rotations from Lemma \ref{imag_rot}. We apply the condition on two terms 
     \begin{align*}
         Q^N|\z1|^{2k}|\z2|^{2l}\z1^{m\alpha}\z2^{m\beta}\quad\text{and}\quad {\overline{Q}}^N|\z1|^{2k}|\z2|^{2l}\z1^{m\alpha}\z2^{m\beta}.
     \end{align*} 
     This gives us two equations 
     \begin{gather*}
         \lambda_1(m\alpha+Np) + \lambda_2(m\beta + Nq ) = 0 \\
         \lambda_1(m\alpha-Np) + \lambda_2(m\beta - Nq ) = 0.
     \end{gather*}
    Adding them reveals $\lambda_1\alpha + \lambda_2\beta = 0$, which in turn provides the formula for $Y$. Since $(\alpha, \beta)$ is independent of $(p,q)$ (otherwise the model would be holomorphically degenerate), the only possibility is $N=0$. The remaining constants we rewrite as $r=\Re(\tau)$.
    At least two terms are needed to have $\g_1=0$ as in the proof of \autoref{gc+real}.
\end{proof}
Using the \autoref{gc+real} and \autoref{gc+imag}, we can recover (a part of) the Lemma~3.4 from \cite{K21}
\begin{theorem} \label{gc+re+im}
If a model has $\g_c\neq0$ and $\dim \g_1 \neq 0$, it is equivalent to 
\begin{align}\label{mod:gc6dim}
    \Im w &= |\z1|^{2k}|\z2|^{2l}\left(\Re \z1^{\alpha}\z2^{\beta}\right)^m, 
\end{align}
where $k,l,m$ are are nonnegative integers, $m>0$ and $\alpha,\beta$ are integers such that $\alpha,\beta\geq-1$.
\end{theorem}
\begin{proof}
    Thanks to \autoref{g1} we have an imaginary rotation that can be diagonalized and we can use the form \eqref{mod:gcim}. This model has to be balanced, hence there is only one summand and we arrive at \eqref{mod:gc6dim}. There is a real rotation given by \eqref{real_rotation}, therefore $\dim \g\geq 6$.
\end{proof}

If a model with nontrivial $\g_c$ has a real and imaginary rotation, we can assume they are both diagonal and apply Theorems \ref{gc+real} and \ref{gc+imag} simultaneously and derive the model \eqref{mod:gc6dim} again. Since all rotations can be diagonalized simultaneously, Theorems \ref{gc+real} and \ref{gc+imag} also say more than two rotations of the same type are not possible. This concludes the discussion of rotations together with $\g_c$.
In the next section we will determine which models from this section possess an additional symmetry.

\section{Exotic together with tubular symmetries}\label{sec:gc+gt}
In the following theorem, we classify models admitting  both  an  exotic symmetry and a tubular symmetry (i.e. symmetry of weight in the interval $(-1,0)$.
\begin{theorem}
\label{gc+tub}
    Any  Levi degenerate model \eqref{model} that is holomorphically nondegenerate and possesses both an exotic symmetry and  a tubular symmetry, i.e. $\g_c\neq0$ and $\g_t\neq0$, is locally equivalent to one of the following models
    \begin{align}
    \Im w &= \Re \bz1\z2^\alpha,                & \dim\g &= 10, \label{mod:zurich}\\
    \Im w &= \Re \bz1\z2^{2\alpha-1} + |\z2|^{2\alpha},   & \dim\g &= 6,\label{mod:gcgtgim}\\
    \Im w & =\Re \bz1\z2^\alpha + Q_1(\z2, \bz2), & \dim\g &= 5,\label{mod:zurichpert}\\
    \Im w &= \Re\z2^\alpha\Re\bz1\z2^\alpha,    & \dim\g &= 5,\label{mod:new_gt+gc}\\
    \Im w &= \Re\z2^\alpha\Re\bz1\z2^\alpha + Q_2(\z2, \bz2), & \dim\g &= 4, \label{mod:new_gt+gcpert}\\
    \Im w &= |\z1|^{2k}(\Re \z2)^m,             & \dim\g &= 7,\label{mod:kolar_gc_7dim}
    \end{align}
    where $\alpha, k, m$ are positive integers and $Q$'s are real-valued polynomials of weighted degree $1$ without pluriharmonic terms such that the corresponding models are not equivalent to those listed previously, e.g., $Q_1$ is not circular ($Q_1 \neq r|\z2|^m$).
\end{theorem}
\begin{remark}
    Structure of $\g$'s is collected in table \ref{tab:result_table}. It can be checked that the multitype weights are equal in all of the cases. Hypersurfaces \eqref{mod:zurich} and \eqref{mod:kolar_gc_7dim} have already been studied in \cite{KoMe11} and \cite{K21}.
\end{remark}
\noindent {\it Proof: }The proof is rather long and we split it into separate lemmas. In all of them we will use the following notation. Recall $\g_t=\gmu1\oplus\gmu2$, denote $T\in\g_t$ a tubular symmetry and $X\in\g_c$ an exotic symmetry. We distinguish several cases according to the commutator $[T,X]$.
First, we let $X$ and $T$ commute:
\begin{lemma}\label{lem:tub1}
    If $[T,X]=0$, then we can choose multitype coordinates such that
    \begin{align*}
        X=i\z2^{\beta}\dz1 &\qaq T = i\dz1 + h\dw, \ \text{or} \\
        X=i\z2^\alpha\dz1 &\qaq T = i\dz2 + h\dw
    \end{align*}
    for some positive integers $\alpha,\beta$ and weighted homogeneous polynomial $h$ of weight $1-\mu_j$.
\end{lemma}
\begin{proof}
We need to distinguish two cases depending on the weight of $T$.

\noindent \textbf{Case I:}
Weight of $T$ is $-\mu_1$, i.e. $T\in\gmu1$.
We will show $X=i\z2^\beta\dz1$.
In general such $T$ is of the form $T=a\dz1+b\dz2+\Tilde{h}\dw$, $a,b\in\bbC$ with $b=0$ if $\mu_1\neq \mu_2$; $\Tilde{h}$ is a polynomial of weighted degree $1-\mu_1$. We can change multitype coordinates so that 
\begin{equation*}
    T=i\dz1 + h\dw \qaq X=f\dz1 + g\dz2,
\end{equation*}
where $f,g,h$ are weighted-homogeneous polynomials. Computation of the commutator
\begin{displaymath}
    0 = [T,X] = if_{\z1}\dz1+ig_{\z1}\dz2-Xh\dw
\end{displaymath}
implies
    $X=\sigma_1 \z2^\alpha\dz1 + \sigma_2\z2^\beta\dz2$,
where $\sigma_1,\sigma_2$ are complex numbers and $\alpha\geq\beta$ (with equality if $\mu_1=\mu_2$). If $\sigma_2$ is nonzero, we can rewrite the field $X$ as 
\begin{align}\label{killing_sigma}
    X=\sigma_2\z2^\beta\left(\sigma\z2^n\dz1+\dz2\right),
\end{align}
where $\sigma=\frac{\sigma_1}{\sigma_2}$ and $n$ is a nonnegative integer. By a change of coordinates 
\begin{equation}\label{change_of_cords}
    \z1=\z1'+\frac{\sigma}{n+1}\z2^{n+1}, \quad \z2=\z2', \quad w = w'
\end{equation}
we transform the field to $X=\sigma_2\z2^\beta\dz2$. This field does not have any $X$-pairs of chains of length more than one. In such a situation the model \eqref{chainsum} would be holomorphically degenerate, which is a contradiction. Hence $\sigma_2=0$ and by re-scaling in the $\z2$ variable we arrive at
\begin{equation*}X=i\z2^\beta\dz1.\end{equation*}

\noindent\textbf{Case II:} 
Weight of $T$ is $-\mu_2$ and $\mu_1>\mu_2$. We claim that  we reach a similar conclusion $X=i\z1^\beta\dz2$. In any multitype coordinates the tubular symmetry $T$ is of the form
\begin{equation}\label{field:tubular-mu2}
T=a\z2^n\dz1+b\dz2+\Tilde{h}\dw.
\end{equation}
Using essentially the same change of coordinates as in \eqref{change_of_cords} we can assume $T=i\dz2+h\dw$.
From $[T,X]=0$ we similarly as before obtain $X=\sigma_1 \z1^\alpha\dz1 + \sigma_2\z1^\beta\dz2$, where $\alpha>\beta$. Although the form looks very similar to \eqref{killing_sigma}, we need to employ a different argument to obtain $\sigma_1=0$. Denote $\nu$ the weight of $X$, then
\begin{equation}
    \nu+\mu_1 = \alpha\mu_1 \quad \text{and} \quad \nu+\mu_2 = \beta\mu_1.
\end{equation}
This in turn implies $\mu_1-\mu_2 = (\alpha-\beta)\mu_1$, which is impossible since $\mu_2\neq 0$. Hence the $\dz1$ term does not appear and we have $X=i\z1^\beta\dz2$.
\end{proof}

The analysis in both cases given by the previous \autoref{lem:tub1} is analogous, so we focus only on the first one. Let us choose coordinates such that
\begin{equation*}
X=i\z2^{\beta}\dz1 \qaq T = i\dz1 + h\dw.
\end{equation*}

Sum of $X$-pairs of chains is given by \autoref{tutu}. In our specific case this takes the form of
\begin{equation}
    \Im w = \sum_{N,K,m} \Re\left(\tau_{N,K,m}\z2^N\right)|\z2|^{2l}\left(\Re\bz1\z2^\beta\right)^m.
\end{equation}
Recall that parameters $N,K,m$ are nonnegative integers, $l=K-m\beta$ and each summand is of weighted degree 1. In the following, we systematically investigate for which choices of parameters this model actually admits the tubular symmetry $T$. To this end, we use the following two lemmas, that are easy observations.
\begin{lemma}\label{lem:tub2}
    Let $S$ be any holomorphic vector field, denote $Z$ its "$\dz{}$-part", i.e.
\begin{equation*}
    S = f\dz1+g\dz2+h\dw = Z + h\dw.
\end{equation*}
The field $S$ is an infinitesimal symmetry of hypersurface $\Im w = P(z,\bz{})$ if and only if $\Re(ZP)$ is pluriharmonic. A real-valued polynomial is pluriharmonic if and only if it does not contain a mixed term (monomial consisting of both holomorphic and antiholomorphic variables).
\end{lemma}
\begin{lemma}
    For any non-constant holomorphic $f$ and real-valued function $R$, the product $R\cdot\Im f$ is pluriharmonic if and only if
$R$ is a real constant or $R=r\Re f$ for some $r\in\bbR$. In the second case
\begin{equation*}
    r\Re f \cdot \Im f = \frac{r}{2}\Im f^2.
\end{equation*}
\end{lemma}

We return to the proof of the \autoref{gc+tub}. Inspired by the previous Lemma \ref{lem:tub2} we compute
\begin{equation}\label{factor_im}
    \Re\left(i\dz1 P\right) = \Im \z2^{\beta}\left(\sum \frac{m}{2}\Re(\tau\z2^N)|\z2|^{2l}\left(\Re\bz1\z2^{\beta}\right)^{m-1}\right).
\end{equation}
If there is a summand with $m>1$, we can identify a mixed term that does not cancel out. Focus on the summand of maximal $N$ among those of maximal $m$, that in turn implies minimal possible $K$. In this summand the monomial 
\begin{equation*}
\z2^{N+K}\bz2^{K-m\beta}\bz1^{m-1}
\end{equation*}
is present and cannot cancel with any other monomial.
Since $N+K>0$ this term is not mixed unless $m=1$. Also, there has to be at least one summand with $m=1$, otherwise the model would be holomorphically degenerate. Next we show that each~$K$ is equal to $\beta$, which means the sum \eqref{factor_im} has only one summand.

Equation \eqref{factor_im} becomes
\begin{equation}\label{factor_im2}
        \Re\left(i\dz1 P\right) = \frac{\Im \z2^{\beta}}{2}\left(\sum_{2K+N=C} \Re(\tau\z2^N)|\z2|^{2(K-\beta)}\right),
\end{equation}
where $C=\beta+\frac{1-\mu_1}{\mu_2}$, $K\in\bbZpos$ and $N\in\bbZnon$. We show the sum is not pluriharmonic by focusing on the term of largest $K$, denoted $K_1$. The corresponding summand is 
\begin{equation*}
\tau\z2^{N_1+K_1-\beta}\bz2^{K_1-\beta} + 
\overline{\tau}\z2^{K_1-\beta}\bz2^{N_1+K_1-\beta}.
\end{equation*}
The first term is mixed if $K_1>\beta$. It could potentially cancel out in the entire sum if there is a term with parameters $K_2, N_2$ satisfying $N_2+K_2 = K_1$ and $K_2 = N_1+K_2$, which is impossible. Hence $K_1 = \beta$ as needed. It follows that models \eqref{mod:zurich} and \eqref{mod:new_gt+gc} are the only possibilities with one $X$-pair of length 2. Their perturbations, i.e., models \eqref{mod:gcgtgim}, \eqref{mod:zurichpert}, and \eqref{mod:new_gt+gcpert}, arise from the presence of additional chains of length $1$. Sums of such $X$-pairs can result in any real polynomial $Q$ in $\z2$. The precise structure of $\g$ for each discovered model needs to be checked directly using simultaneous diagonalization of $\g_0$ and general form of tubular symmetries. Model \eqref{mod:gcgtgim} arises as the only perturbation possessing an imaginary rotation. This concludes the proof for the situation when $X$ and $T$ commute. 

Now we focus on noncommuting $T$ and $X$. Since $\dim\g_c=1$ and $\g_n=0$, the weight of the bracket is nonpositive. First, let the commutator be of weight 0, i.e., $[T,X] \in \g_0$.

\begin{lemma}
    If $[T,X]=Y\neq0$ is rotation, then there are coordinates such that
    \begin{equation}\label{field:quadratic}
    X = i\z1\left(q\z1\dz1-p\z2\dz2\right).
\end{equation}
\end{lemma}
\begin{proof}
By \autoref{nonil} and rigidity of $T$ and $X$ we have $Y=Y^{\Re}+Y^{\Im}$. 
We can change multitype coordinates such that $Y$ is diagonal and by \autoref{diagrots} we can assume $X$ to be monomial in these coordinates
\begin{equation*}
X = i\z1^\alpha\z2^\beta\left(q\z1\dz1-p\z2\dz2\right).
\end{equation*}
First, 
let the weights be equal. The general form of a tubular symmetry is
\begin{equation*}
T=a\dz1+b\dz2+h(z)\dw.
\end{equation*}
The weight of $X$ is $\mu_1$, hence its coefficients are quadratic, i.e. $\alpha+\beta = 1$. The choice $\alpha=-1$ leads to the field $X=i\z2^2\dz1$, but in this case the commutator
\begin{equation*}
    [T,X] = 2ib\z2\dz1
\end{equation*}
is not diagonal. Therefore (after interchanging variables if needed) we arrive at \eqref{field:quadratic}. If $\mu_1>\mu_2$, the reader can compute directly $[T,X]$ using the general form of $T$ given by \eqref{field:tubular-mu2} similarly as above and observe we always end up with \eqref{field:quadratic}.
\end{proof}

We can assume (after possible switching of variables) the exotic symmetry is given by \eqref{field:quadratic}. From the commutator condition
\begin{equation}\label{commutator_is_rotation}
    Y = [T,X] =\lambda_1\z1\dz1+\lambda_2\z2\dz2 =ia\left(2q\z1\dz1-p\z2\dz2\right)-ibp\z1\dz2-Xh\dw,
\end{equation}
it follows that $b=0$ and $Xh=0$.
If $Y^{\Re} \neq 0$, the description from the \autoref{gc+real} takes the form
\begin{equation*}
    \Im w = \sum_{2K+N = C} \Re\left(\tau_{K,N}Q^N\right)|\z1|^{2(Kp-m)}|\z2|^{2Kq}\left(\Re\z1\right)^m,
\end{equation*}
where $Q = \z1^p\z2^q$ and $m$ is fixed. We show that the expression has only one summand. Comparing the formula \eqref{commutator_is_rotation} with \eqref{real_rotation} we obtain
\begin{equation*}
    \frac{p}{2q}=\frac{Cp-m}{Cq}.
\end{equation*}
It follows $m=\frac{Cp}{2}$. The exponent of the $|\z1|^2$ factor is now $Kp-m=Kp-\frac{(2K+N)p}{2} = -\frac{N}{2}\geq 0$, hence $N=0$ and $Kp=m$. We have retrieved the model \eqref{mod:kolar_gc_7dim}. If $Y^{\Re}=0$, $Y$ would be an imaginary rotation. A comparison with \eqref{imaginary_rotation} shows this situation does not occur, since the rotation $Y$ has both $\dz1$ and $\dz2$ terms, while $\beta=0$. This concludes the case when $[T,X] \in \g_0$.

The only remaining case is $[T,X]=T_2$, where $T_2$ is also a tubular symmetry necessarily of higher weight. Now we can consider the commutator $[T_2,X]$ already covered in previous cases and conclude that no additional models are possible.\qed

\section{Solitary exotic symmetry}\label{sec:gc3dim}

A procedure for verifying that a given model admits an exotic symmetry is described in \cite{KM16}. Such an $M$ admits no other symmetry if and only if it is not of the equivalent to those given in Theorems \ref{gc+real},  \ref{gc+imag} and \ref{gc+tub}. In particular, by inspecting the proofs we notice the exotic symmetry is monomial and we obtain the following
\begin{proposition}
    Let a holomorphically nondegenerate model \eqref{model} possess a nontrivial $\g_c$, generated by a vector field  $X$. 
    If $X$ is not a monomial multiple of a linear diagonal vector field in any multitype coordinates, then there is no additional symmetry and $\dim \g = 3$.
\end{proposition}

The results of Section 3 also provide a  characterization in the case when $X$ is a monomial multiple of a linear diagonal vector field in some multitype coordinates. A model of the form  

\eqref{mod:monomial_gc} 
has a 3-dimensional symmetry algebra if and only if it is not of one of the special forms  
\eqref{mod:gcre} or \eqref{mod:gcim}.

\begin{example}
We now give an example of a model for which $X\in \g_c$ is a monomial multiple of a linear diagonal vector field in some multitype coordinates, and $\dim \g=3$. In fact, at the same time it provides the first example of a model with nontrivial $\g_c$ for which the Catlin multitype weights are different.

We consider $U_1= {z_1}^4 + {z_2}^2 {z_1}^3, \ \  U_2= 2 {z_2} {z_1}^8,  \ \  U_3= 2{z_1}^{13},$ with a vector field $X$ given by 
$X= i{z_1}^5\partial_{z_2}.$ The weights are $\mu_1 = \frac{1}{17}, \ \mu_2 = \frac{1}{34}.$
Hence the model is given by
\begin{equation*}
    \Im w = 8 |\z1|^6\left(\Re\z1^5\bz2\right)^2 +  4 |\z1|^{8}\Re{\z1^9}.
\end{equation*}

\end{example}

The following result was proved in \cite{KM17}.
\begin{theorem} \label{th2}There exists a non-monomial model $M$ such that $\dim
\g_c
=1$ and $\dim
\g
=3$.
\end{theorem}
The model which proves the statement is provided by
\begin{equation}
P(z_1, z_2) = i z_1^2 z_2^3 (z_1 - z_2),
\quad
Q(z_1, z_2) =  3 z_1^3 z_2^5 (z_1 - z_2)
\end{equation}
and a vector field
\begin{equation}
 X =z_1 z_2^2 (5z_1 - 6z_2) \partial_{z_1}  - z_2^3 (4z_1 - 3 z_2)
\partial_{z_2}.
\end{equation}
It is easy to check that $X(P) = i Q$ and $X(Q) = 0$, and
therefore $X$ is an exotic symmetry for $M$ given by $ \Im w =\Re P \bar Q$. Since the coefficients of $X$ vanish along three different complex lines, it follows that $X$ is not a monomial multiple of a linear diagonal vector field in any multitype coordinates.

\end{document}